\documentclass[11pt,a4paper,reqno]{amsart}
\usepackage{graphicx,amsmath,amsfonts,amssymb,amsthm} 

\calclayout

\usepackage[
backend=biber,
style=alphabetic,
sorting=ynt,
maxbibnames=10,
maxalphanames=4
]{biblatex}

\usepackage{color}
\usepackage[colorlinks=true, linkcolor=blue, citecolor=red]{hyperref}
\usepackage{cleveref}
\usepackage{tabularx}
\usepackage{float}
\usepackage{nicefrac}
\usepackage[]{appendix}
\usepackage{tcolorbox}
\usepackage{xcolor}
\hypersetup{
    colorlinks,
    linkcolor={blue!60!black},
    citecolor={red!60!black},
    urlcolor={blue!80!black}
}

\numberwithin{equation}{section}

\newtheorem{theorem}{Theorem}[section]
\newtheorem*{theorem*}{Theorem}

\newtheorem{lemma}{Lemma}[section]
\newtheorem{prop}[lemma]{Proposition}

\newtheorem{corollary}[lemma]{Corollary}

\theoremstyle{definition}
\newtheorem{definition}[lemma]{Definition}

\theoremstyle{remark}
\newtheorem{remark}[lemma]{Remark}

\newcommand{\be}{\begin{equation}}
\newcommand{\ee}{\end{equation}}
\newcommand{\bu}{\mathbf{u}}

\newcommand{\br}{\mathbf{r}}
\renewcommand{\d}{\, \mathrm{d}}
\newcommand{\der}{\mathrm{d}}

\renewcommand{\leq}{\leqslant}
\renewcommand{\geq}{\geqslant}

\newcommand*{\defeq}{\mathrel{\vcenter{\baselineskip0.5ex \lineskiplimit0pt
                     \hbox{\scriptsize.}\hbox{\scriptsize.}}}%
                     =}

\newcommand*{\eqdef}{=\mathrel{\vcenter{\baselineskip0.5ex \lineskiplimit0pt
                     \hbox{\scriptsize.}\hbox{\scriptsize.}}}%
                     }

\title[Carrollian Fluids in One Dimension]{One-Dimensional Carrollian Fluids II: $C^1$ Blow-Up Criteria}

\author[N.~Athanasiou]{Nikolaos Athanasiou}
\address[N.~Athanasiou]{Department of Mathematics, University of Crete,
Voutes Campus, 70013 Heraklion, Greece}
\email{n.athanasiou@uoc.gr}

\author[P.~M.~Petropoulos]{Marios Petropoulos}
\address[P.~M.~Petropoulos]{Centre de Physique Th\'eorique, Ecole Polytechnique, 91120 Palaiseau, France}
\email{marios.petropoulos@polytechnique.edu}

\author[S.~M.~Schulz]{Simon Schulz}
\address[S.~M.~Schulz]{Scuola Normale Superiore, Centro De Giorgi, P.zza dei Cavalieri, 3,  56126 Pisa, Italy}\email{simon.schulz@sns.it}

\author[G.~Taujanskas]{Grigalius Taujanskas}
\address[G.~Taujanskas]{Faculty of Mathematics, Wilberforce Road,
Cambridge CB3 0WA,
UK}
\email{taujanskas@dpmms.cam.ac.uk}



\thanks{\emph{Centre de Physique Th\'eorique Preprint Number.} CPHT-RR027.052024.}

\begin{document}
\begin{abstract}
    The Carrollian fluid equations arise from the equations for relativistic fluids in the limit as the speed of light vanishes, and have recently experienced a surge of interest in the theoretical physics community in the context of asymptotic symmetries and flat-space holography. In this paper we initiate the rigorous systematic analysis of these equations by studying them in one space dimension in the $C^1$ setting. We begin by proposing a notion of \emph{isentropic} Carrollian equations, and use this to reduce the Carrollian equations to a $2 \times 2$ system of conservation laws. Using the scheme of Lax, we then classify when $C^1$ solutions to the isentropic Carrollian equations exist globally, or blow up in finite time. Our analysis assumes a Carrollian analogue of a \emph{constitutive relation} for the Carrollian energy density, with exponent in the range $\gamma \in (1, 3]$.
\end{abstract}

\maketitle
\thispagestyle{empty}


\setcounter{tocdepth}{1}
\tableofcontents

\section{Introduction}

This paper is the second in the series \cite{Duality,Carrollcomp} in which we initiate the systematic study of the well-posedness theory of the flat Carrollian fluid equations. This work is concerned with the $C^1$ theory of the one-dimensional \emph{isentropic}\footnote{See \S \ref{sec:presentation} for an explanation of the terminology \emph{isentropic}. Briefly, the system \eqref{eq:carroll eqns i} is the dual of the isentropic Galilean Euler equations under the Carroll--Galilei duality mapping of \cite{Duality}.} \emph{Carrollian fluid equations}
\begin{equation}\label{eq:carroll eqns i}
    \left\lbrace\begin{aligned}
        &\partial_t (\sigma \beta) + \partial_x \sigma = 0, \\ 
        & \partial_t \left(\gamma^{-1} \sigma^\gamma + \sigma \beta^2\right) + \partial_x (\sigma \beta) = 0, 
    \end{aligned}\right.
\end{equation}
posed in $(t,x) \in (0,\infty) \times \mathbb{R} \eqdef  \mathbb{R}^2_+$, where $\gamma \in (1,3]$; the quantity $\sigma$ is called the \emph{Carrollian stress}, with the term $\gamma^{-1}\sigma^\gamma$ being the \emph{Carrollian internal energy density} corresponding to a polytropic constitutive relation in the dual Galilean setting (\textit{cf.}~\cite{Duality}), and $\beta$ is the \emph{Carrollian velocity}. Using the classical method of Lax (\textit{cf}.~\cite{Lax}), we establish necessary and sufficient conditions for the finite-time blow-up of solutions that are initially continuously differentiable. Our main result is the following (see \S\ref{sec:main results} for detailed statements).

\begin{theorem*}
    Let $\gamma \in (1, 3]$ and let $(\sigma_0, \beta_0) \in C^1(\mathbb{R})$ be admissible initial data for \eqref{eq:carroll eqns i}. Then there exists a unique global classical solution $(\sigma, \beta) \in C^1(\mathbb{R}^2_+)$ to \eqref{eq:carroll eqns i} if and only if the initial data is everywhere rarefactive. Conversely, if the initial data is compressive somewhere, the local solution ceases to be continuously differentiable in finite time $T^*$, which we control quantitatively in terms of the initial data.
\end{theorem*}

The terms \emph{rarefactive} and \emph{compressive} are explained in \S \ref{sec:compression rarefaction}.  The results of the present paper motivate the need for the companion paper \cite{Carrollcomp}, where the authors establish global-in-time existence in a less regular functional setting for the exponent $\gamma=3$ (which is the only instance where \eqref{eq:carroll eqns i} can be rewritten as a conservative system for $(\sigma,\beta)$; see \cite[\S 2]{Carrollcomp}) using the theory of compensated compactness.

\subsection{Physical and mathematical context} The Carrollian fluid equations are formally derived as the limit of the relativistic fluid equations $\nabla_a T^{ab} = 0$, where $T_{ab}$ is the standard energy-momentum tensor for relativistic fluids, when the speed of light $c$ vanishes \cite{CiambelliMarteauPetkouPetropoulosSiampos2018,PetkouPetropoulosRiveraBetancourSiampos2022,Duality,Hartong2015,deBoerHartongObersSybesmaVandoren2018}. The $c \to 0$ limit (the Carrollian limit) of the Poincar\'e group goes back to the work of L\'evy-Leblond \cite{LevyLeblond1965} and Sen Gupta \cite{SenGupta1966} in the 60s; aside from the contribution of Henneaux \cite{Hennaux1979}, however, these ideas appear to have been dormant until recently, when Carrollian geometry and physics experienced a renewed interest in the mathematical physics community in the context of asymptotic symmetries and  flat-space holography \cite{DuvalGibbonsHorvathyZhang2014,DuvalGibbonsHorvathy2014,DuvalGibbonsHorvathy2014b,BekaertMorand2016,BekaertMorand2018,Morand2020,CiambelliLeighMarteauPetropoulos2019,Herfray2022}. In $d+1$ spacetime dimensions the Carrollian limit gives a degenerate $(d+1)$-dimensional metric. Degenerate metrics appear on embedded null hypersurfaces as well as on null boundaries of $(d+2)$-dimensional asymptotically flat spacetimes, obtained in the vanishing cosmological constant limit from asymptotically $\mathrm{AdS}_{d+2}$. Through a correspondence between fluid observables on the boundary and gravity in the bulk, Carrollian fluids are then seen to encode elements of the holographic duals of the gravitational field in an asymptotically flat spacetime.

Even in the simplest\footnote{We study non-perfect Carrollian fluids which are dual to perfect Galilean fluids; see \S\ref{sec:presentation} and \cite{Duality}.} case of Carrollian spacetimes with $d=1$, equipped with a flat connection, however, no rigorous mathematical theory appears to exist for the Carrollian fluid equations. In this case the Carrollian limit of the relativistic fluid equations yields three conservation laws in five observables: the Carrollian velocity $\beta$, the Carrollian stress $\sigma$, a generalized Carrollian pressure $\varpi$, an internal energy density $\epsilon$, and a Carrollian heat current $\pi$, which is assumed to be given.\footnote{The generalized pressure $\varpi$ is the total stress, \emph{i.e.} the equilibrium pressure plus the viscous stress. Accordingly, the Carrollian stress $\sigma$ should have been referred to as \emph{superstress} because it appears at higher order in the $c^{-1}$ expansion of the relativistic stress (see \cite{Duality}). Together with the Carrollian heat current $\pi$, the data $\varpi$ and  $\sigma$ betray the remoteness from an ideal fluid. Be this as it may, this nomenclature remains formal, as to date there is no thermodynamic, microscopic or kinetic theory for Carrollian fluids and concepts such as a continuous medium or the entropy are just borrowed from Galilean physics without \emph{ab initio} definition (attempts can be found in \cite{deBoerHartongObersSybesmaVandoren2023}).} This gives a $3 \times 2$ system of conservation laws. In $1+1$ dimensions it is known that the Carrollian and Galilean symmetry groups are isomorphic, and in fact there exists a duality map between the Carrollian fluid equations and the three Galilean compressible Euler equations \cite{Duality}. This duality map interchanges time and space, and maps Carrollian observables to their Galilean counterparts nonlinearly (see \Cref{tab:duality_table}). We use the Carroll--Galilei duality to trivialize the equation for $\varpi$, by postulating the constitutive relation $\epsilon = \gamma^{-1} \sigma^\gamma$ for the internal energy density and an expression for $\varpi$, borrowing directly by duality from the Galilean expressions for ideal polytropic gases. This reduces the system to a $2\times 2$ system of conservation laws for $\sigma$ and $\beta$, which we suggest to call the \emph{isentropic Carrollian fluid equations}. The details of this reduction are given in \S\ref{sec:presentation}. The resulting system is precisely \eqref{eq:carroll eqns i}, and is the focus of the analysis of this paper.

Our goal is to study the Cauchy problem for \eqref{eq:carroll eqns i}. A first natural question to ask concerns the character of the equations. Indeed, individual motion ought to be prohibited in the Carrollian limit, which might suggest that the system \eqref{eq:carroll eqns i} should represent a stationary system. This is not the case, however, both from a physical perspective, as collective Carrollian phenomena may depend on time (see \emph{e.g.} \cite{BergshoeffGomisLonghi2014}), or in the sense of the classification of PDEs. Indeed, it can be seen by duality that the equations are hyperbolic at least in some region of phase space, since the isentropic compressible Euler equations are hyperbolic on the Galilean side. Nevertheless, knowledge of the Galilean solutions and the existence of the duality mapping does not provide information about the Carrollian Cauchy problem due to the fact that the Galilean Cauchy problem is mapped to a \emph{boundary} value problem under duality. The Carrollian Cauchy problem therefore has to be studied anew. As we will show, the isentropic Carrollian fluid equations \eqref{eq:carroll eqns i} are a genuinely nonlinear $2 \times 2$ system of hyperbolic conservation laws. Hyperbolic conservation laws have an extremely rich history in the literature, see \textit{e.g.}~\cite{Dafermos2016} and the references therein; in particular, it is classical that in evolution such systems may develop shocks, leading to losses of regularity. Shocks may be exhibited through a method of Lax \cite{Lax}, which has been used both in cases of $2 \times 2$ hyperbolic systems as well as nonlinear wave equations: see for instance the classical references \cite{JohnSing,LiuSing,KlainSing,Sideris}. For completeness, we recap Lax's method in \S\ref{sec:lax background}. In the context of Galilean fluids, similar techniques have been used to treat both the one-dimensional isentropic compressible Euler equations\footnote{In its Lagrangian form, the so-called $p$-system.} \cite{Geng1} and the one-dimensional non-isentropic compressible Euler equations \cite{Geng2}, as well as the three-dimensional spherically symmetric equations of magneto-hydrodynamics \cite{Geng3}. The \emph{relativistic} Euler equations have been studied in \cite{NikosSheng1,AthTianZhu}; the former contribution treats the isentropic relativistic Euler system (giving necessary and sufficient conditions for singularity formation), while the latter provides sufficient conditions for blow-up of the non-isentropic system. In contrast to the Galilean Euler equations, the relativistic equations admit no obvious Lagrangian formulation, which makes their analysis rather different; in particular, necessary conditions for singularity formation in the non-isentropic setting are currently out of reach. For further details and background, we refer the reader to the recent doctoral thesis\footnote{Bayles-Rea's thesis is devoted entirely to the theory of $C^1$ singularity formation for the compressible Euler equations and treats in particular the isentropic and non-isentropic cases, as well as the classical and relativistic formulations.} of Bayles-Rea \cite{TianruiThesis}.

\subsection{Results}

In this paper we prove that the system \eqref{eq:carroll eqns i} does in fact develop shocks in the $C^1$ setting, and classify when they occur. More precisely, we classify initial data which preserve the hyperbolicity of the system for all times, and use the method of Lax \cite{Lax} to establish necessary and sufficient conditions on the initial data for $C^1$ blow-up\footnote{Sharp criteria for blow-up of the Lipschitz, H\"older, and $BV$ norms will be the subject of future investigations.} in finite time for all exponents $\gamma \in (1,3]$, the so-called Galilean physical range. The case $\gamma = 1$ corresponds to a linear pressure term, the analysis of which is different and will be the focus of future works; on the Galilean side this corresponds to the case of an isothermal gas. Along the way we provide a robust definition of compression and rarefaction for the system \eqref{eq:carroll eqns i}.

Our analysis is based on obtaining Riccati-type equations for the evolution of Riemann invariants along characteristic curves. In harmony with results on the Galilean side (e.g. \cite{Geng1,NikosSheng1,AthTianZhu}), we prove that (any) compression is a necessary and sufficient condition for singularity formation in finite time. We note that here the usual notions of \emph{compression} and \emph{rarefaction} must be refined due to the relative signs of the Riemann invariants and the derivatives of the eigenvalues of the system: see \S \ref{sec:compression rarefaction}. As explained also in the companion paper \cite{Carrollcomp}, we furthermore find that the system \eqref{eq:carroll eqns i} exhibits a novel type of degeneracy, namely that the eigenvalues of the flux matrix degenerate along entire curves $\left\{\beta = \pm \sigma^{(\gamma-1)/2} \right\}$ in the phase plane, leading to a loss of strict hyperbolicity on these loci (as well as the typical degeneracy at the \emph{Carrollian liquescence}\footnote{In the Galilean setting, the locus $\{\rho=0\}$ is called the vacuum (or cavitation); however, for a Carrollian fluid, $\sigma=0$ is representative of the fluid becoming inviscid, which we propose to call \emph{Carrollian liquescence}.} $\{ \sigma = 0 \}$). A core part of this work involves establishing invariant regions in phase space which constrain the system to remain strictly hyperbolic; in the special case $\gamma=3$, these invariant regions are the same as the ones established in the companion paper \cite[\S 6]{Carrollcomp}, although they are obtained using a different approach.

\subsection{Plan of the paper.} In \S \ref{sec:core} we introduce the isentropic Carrollian fluid system and explain the duality with the compressible Euler equations (\S \ref{sec:presentation}), state our main theorems (\S \ref{sec:main results}), and provide a short summary of the method of Lax for $C^1$ blow-up of general strictly hyperbolic systems (\S \ref{sec:lax background}). Then, in \S \ref{sec:hyp}, we establish the hyperbolicity and genuine nonlinearity of the system in the relevant regions of the phase space. In \S \ref{sec:invariant} we establish conditions on initial data which constrain the solutions to the hyperbolic regions of phase space for all times, and in \S \ref{sec:local} we establish the local-in-time existence and uniqueness of $C^1$ solutions. Finally, \S \ref{sec:gamma 3} is concerned with the $C^1$ blow-up criterion for the case $\gamma=3$, while \S \ref{sec:blow up} is concerned with the general case $\gamma \in (1,3)$.

\section{Core Notions and Main Results}\label{sec:core}

\subsection{Isentropic Carrollian Equations.}\label{sec:presentation} The non-perfect Carrollian fluid equations on a flat background with one space dimension are the $3 \times 2$ system 
\begin{equation}\label{eq:carroll_eqns_0}
    \left\lbrace\begin{aligned}
        &\partial_t (\beta \sigma) + \partial_x \sigma = 0, \\ 
        & \partial_t (\epsilon + \beta^2 \sigma) + \partial_x (\beta \sigma) = 0, \\
        & \partial_t (\beta \varpi) + \partial_x \varpi = - \partial_t (\beta \epsilon  + \pi), 
    \end{aligned}\right.
\end{equation}
where we assume that the \emph{internal energy density} $\epsilon$ satisfies the constitutive relation $\epsilon = \gamma^{-1}\sigma^\gamma$, borrowing this constitutive relation directly by duality from the usual Galilean constitutive relation $p = \gamma^{-1} \rho^\gamma$. The unknowns $\sigma$ and $\beta$ are the Carrollian stress and Carrollian velocity, as described in the introduction (see also \cite[Eqs.~(94)--(96)]{Duality}). The third unknown, $\varpi$, we refer to as the \emph{generalized Carrollian pressure}, while $\pi$ is a \emph{Carrollian heat current} and is assumed to be given. We highlight that, assuming the constitutive relation for $\epsilon$, the first two equations decouple and form an autonomous system since they do not involve the third unknown $\varpi$. 

Under the Carroll--Galilei duality mapping, the system \eqref{eq:carroll_eqns_0} is dual to the full compressible Euler system of Galilean perfect-fluid mechanics,\footnote{As explained in \cite{Duality}, the Carroll--Galilei duality interchanges the longitudinal and transverse directions, permuting therefore equilibrium and out-of-equilibrium observables. This explains why a Galilean perfect fluid is mapped onto a non-perfect Carrollian one.} \textit{i.e.} 
\begin{align}
    \label{full_galilean_euler_1}
   & \partial_t \rho + \partial_x (\rho v ) = 0, \\
    \label{full_galilean_euler_2}
    & \partial_t (\rho v) + \partial_x (\rho v^2 + p ) = 0, \\
    \label{full_galilean_euler_3}
    & \partial_t \Big(\rho e + \frac{1}{2} \rho v^2 \Big) + \partial_x \Big( \rho e v + \frac{1}{2} \rho v^3 + p v \Big) = 0, 
\end{align}
where $\rho$ is the mass density of the fluid, $v$ its velocity, $p$ the pressure and $e$ the internal specific energy. Supplemented with the constitutive relation and the energy equipartition law for an ideal polytropic gas 
\begin{equation}
\label{isentropic_Galilei_pressure} 
 p  = K \text{e}^{\nicefrac{s}{c_v}} \rho^\gamma \qquad \text{and} \qquad e = \frac{1}{\gamma - 1} \frac{p}{\rho},
 \end{equation}
where $\gamma > 1$, $K$ is a dimensionful positive constant, $c_v$ is the specific thermal capacity of the gas, and $s = s(t,x)$ is the specific entropy of the system (\textit{cf.}~\cite{Landau}), a standard calculation (\textit{cf.}~\textit{e.g.}~\cite[\S 1]{ChenChenZhu2019}) using the product rule shows that for classical solutions of \eqref{full_galilean_euler_1}--\eqref{full_galilean_euler_3} these assumptions reduce the Galilean time-energy equation \eqref{full_galilean_euler_3} to the stationarity of entropy, 
\[ \partial_t s = 0. \]
The fluid being ideal in the case under consideration, the entropy is conserved, \emph{i.e.} $\nicefrac{\text{d}s}{\text{d}t}=0$, hence $s$ is a constant. The system \eqref{full_galilean_euler_1}--\eqref{full_galilean_euler_3}
reduces to the $2 \times 2$ system of \emph{isentropic Euler equations}
\begin{equation*}
 \left\lbrace \begin{aligned}  & \partial_t \rho + \partial_x (\rho v ) = 0, \\
    & \partial_t (\rho v) + \partial_x (\rho v^2 + p ) = 0,
    \end{aligned}\right.
\end{equation*}
where, without loss of generality (from now on the dynamical variables are dimensionless), 
\begin{equation} \label{ideal_polytropic_gas} p = \frac{1}{\gamma} \rho^\gamma \qquad \text{and} \qquad e = \frac{1}{\gamma - 1} \frac{p}{\rho} = \frac{1}{\gamma(\gamma-1)} \rho^{\gamma-1}. \end{equation}
Via duality with the Galilean equations \eqref{full_galilean_euler_1}--\eqref{full_galilean_euler_3}, this provides a natural way of reducing the full Carrollian system \eqref{eq:carroll_eqns_0} to a $2 \times 2$ system. We recall that the duality map exchanges time and space, the Carrollian stress $\sigma$ and Galilean density $\rho$, the Carrollian and Galilean velocities $\beta$ and $v$, the generalized Carrollian pressure $\varpi$ with the total Galilean energy density, and the Carrollian internal energy density $\epsilon$ with the generalized Galilean pressure, which coincides with the thermodynamic pressure due to the ideal nature of the Galilean fluid, as summarised in \Cref{tab:duality_table}.

\begin{table}[H]
    \begin{center}
    \def\arraystretch{1.3}
    \begin{tabularx}{0.633\textwidth}{ | c | c | c | c | c | c | c | } 
    \hline
    Galilean variable & $t$ & $x$ & $\rho$ & $v$ & 
    $p= \gamma^{-1} \rho^\gamma$ & $\rho(e + \frac{1}{2}v^2)$
    \\ 
    \hline
    Carrollian variable & $x$ & $t$ & $\sigma$ & $\beta$ & $\epsilon = \gamma^{-1} \sigma^\gamma$ & $\varpi$ \\ 
    \hline
    \end{tabularx}
    \end{center}
    \caption{A summary of relevant Galilean variables and their Carrollian duals.}
    \label{tab:duality_table}
\end{table}

The ideal polytropic gas assumptions \eqref{ideal_polytropic_gas} then suggest the following dual Carrollian internal energy density and generalized Carrollian pressure
\begin{equation} \label{isentropic_generalized_Carroll_pressure} 
 \epsilon = \frac{1}{\gamma} \sigma^\gamma \qquad \text{and} \qquad 
\varpi = \frac{1}{2} \sigma \beta^2 + \frac{1}{\gamma (\gamma-1)} \sigma^\gamma, \end{equation}
which can be checked by direct computation to formally satisfy the Carrollian space-momentum equation
\[ \partial_t (\beta \varpi ) + \partial_x \varpi = - \partial_t (\beta \epsilon + \pi) \]
with
\[ \pi = 0 \]
as a consequence of the first two Carrollian equations in \eqref{eq:carroll_eqns_0}. The first of \eqref{isentropic_generalized_Carroll_pressure} could be more accurately expressed as the first equation in \eqref{isentropic_Galilei_pressure}:
\[ \epsilon  = \tilde{K} \text{e}^{\nicefrac{\tilde{s}}{\tilde{c}_v}} \sigma^\gamma,  \]
with $\tilde{K}$ some dimensionful constant and $\tilde{s}$, $\tilde{c}_v$ two thermodynamic-like Carrollian variables. It is tempting to interpret $\tilde{s}$, \emph{i.e.}
\[ \tilde{s} = \tilde{c}_v \log \left( \frac{\epsilon \sigma^{-\gamma}}{\tilde{K}} \right), \]
as a \emph{Carrollian entropy}, remaining constant in the evolution of the system at hand. The latter will thus be referred to as \emph{an isentropic Carrollian fluid.}

We therefore make the following definition.
\begin{definition}[Isentropic Carrollian Equations]
    We call the system \eqref{eq:carroll eqns i}, \textit{i.e.} 
    \begin{equation*}
    \left\lbrace\begin{aligned}
        &\partial_t (\beta \sigma) + \partial_x \sigma = 0, \\ 
        & \partial_t (\gamma^{-1} \sigma^\gamma + \beta^2 \sigma) + \partial_x (\beta \sigma) = 0
    \end{aligned}\right.
    \end{equation*}
    the \emph{isentropic Carrollian fluid equations}. 
\end{definition}
The discussion above implies that, for $C^1$ solutions of \eqref{eq:carroll eqns i}, the expression \eqref{isentropic_generalized_Carroll_pressure} for $\varpi$ trivially satisfies the Carrollian space-momentum equation with $\pi = 0$. From now on we therefore restrict our attention to the analysis of the isentropic system \eqref{eq:carroll eqns i}. The $C^1$ theory of the full non-isentropic system \eqref{eq:carroll_eqns_0} will be the subject of future investigations. 

\subsection{Notions of compression and rarefaction}\label{sec:compression rarefaction}

Before presenting the main results, we provide a definition of compression and rarefaction; this will make the content of our theorems more intuitive. The underlying idea is the notion that \emph{compression} ought to describe the tendency of characteristic curves to fall onto each other (and thus giving rise to a loss of regularity), while \emph{rarefaction} ought to describe the tendency of characteristic curves to move away from one another.  In this section we illustrate that compression and rarefaction are determined from the signs of the spatial derivatives of the eigenvalues.

Recall that, for a general hyperbolic system, the $i$-th Riemann invariant $w_i$ is constant along the $i$-th characteristic curve ($i=1,2$), and that there holds 
\begin{equation*}
   \left\lbrace \begin{aligned}
        &\partial_t w_1 + \lambda_2 \partial_x w_1 = 0, \\ 
        &\partial_t w_2 + \lambda_1 \partial_x w_2 = 0, 
    \end{aligned}\right. 
\end{equation*}
with $\lambda_j$ ($j=1,2$) the eigenvalues of the system (\textit{cf.}~\textit{e.g.}~\cite[\S 11.3.1, Theorem 1]{Evans}). Focusing on the set of characteristics associated to $w_1$ for the time being, the characteristic emanating from the point $x_0\in\mathbb{R}$, denoted by $X(t,x_0)$, satisfies 
\begin{equation}\label{eq:char curve illustration}
   \left\lbrace \begin{aligned}
        &\frac{\der X(t,x_0)}{\d t} = \lambda_2(t,X(t,x_0)), \\ 
        &X(0,x_0) = x_0. 
    \end{aligned} \right. 
\end{equation}
Let us assume that the characteristics intersect in finite time (\textit{i.e.}~compression): let $x_0 < y_0$ and suppose that $t_*$ is the first time at which the curves $X(t,x_0)$ and $X(t,y_0)$ intersect, \textit{i.e.}~we assume that 
\begin{equation}
\label{comp_rar_explanation_assumption}
    X(t,x_0) < X(t,y_0) \quad \text{for all } t \in [0,t_*) , \quad \text{and } \quad X(t_*,x_0) = X(t_*,y_0). 
\end{equation}
Using the Fundamental Theorem of Calculus and \eqref{eq:char curve illustration} to rewrite $X(t_*,x_0) = X(t_*,y_0)$ in terms of $\lambda_2$, we get 
\begin{equation*}
  - \int_0^{t_*} \Big( \lambda_2(s,X(s,y_0)) -\lambda_2(s,X(s,x_0)) \Big) \d s = y_0 - x_0 > 0. 
\end{equation*}
By the Mean Value Theorem, for all $s \in [0,t_*]$ there exists $\xi(s) \in (X(s,x_0),X(s,y_0))$ such that 
\begin{equation*}
    \lambda_2(s,X(s,y_0)) -\lambda_2(s,X(s,x_0)) = \lambda_{2x}(s,\xi(s)) \big( X(s,y_0) - X(s,x_0) \big), 
\end{equation*}
whence 
\begin{equation}\label{eq:characteristics cross illustration}
  - \int_0^{t_*} \lambda_{2x}(s,\xi(s)) \underbrace{\big( X(s,y_0) - X(s,x_0) \big)}_{>0 \text{ by \eqref{comp_rar_explanation_assumption}}} \d s > 0. 
\end{equation}
We deduce that if $\lambda_{2x} \geq 0$ everywhere in $[0,t_*]\times\mathbb{R}$ then we contradict \eqref{eq:characteristics cross illustration}. In turn, we see that compression arises where $\lambda_{2x} < 0$. An analogous computation for the second set of characteristic curves shows that compression also arises where $\lambda_{1x} < 0$. This motivates the following definition.

\begin{definition}[Compressive and Rarefactive Solutions]\label{def:compressive rarefactive} Let $\lambda_1\leq \lambda_2$ be the two eigenvalues of the system \eqref{eq:carroll eqns i}. A solution of \eqref{eq:carroll eqns i} at a point $(t,x)$ is called: 
\begin{enumerate}
    \item[(i)] Forward rarefactive (FR) if $\lambda_{2x}(t,x) \geq 0;$
    \item[(ii)] Backward rarefactive (BR) if $\lambda_{1x}(t,x) \geq 0;$
    \item[(iii)] Forward compressive (FC) if $\lambda_{2x}(t,x)<0;$
    \item[(iv)] Backward compressive (BC) if $\lambda_{1x}(t,x)<0.$
\end{enumerate}
We say that a solution is \emph{compressive} at a point $(t,x)$ if it is either FC or BC at $(t,x)$, and we say that a solution is \emph{rarefactive} at a point $(t,x)$ if it is either FR or BR at $(t,x)$. 
\end{definition}

\subsection{Main Results}\label{sec:main results}

We begin by introducing our notion of solution, which is that of a classical solution with $C^1$ regularity; such a solution satisfies the equations \eqref{eq:carroll eqns i} in the pointwise sense. 

\begin{definition}[$C^1$ Solution]
Let $T>0$ and $(\sigma_0,\beta_0) \in C^1(\mathbb{R})$ with $\sigma_0 \geq 0$ on $\mathbb{R}$. The pair $(\sigma,\beta)$ is called a \emph{$C^1$ solution to the Carrollian equations} \eqref{eq:carroll eqns i} on $(0,T)\times \mathbb{R}$ if: 
\begin{itemize}
    \item[(i)] $\sigma,\beta \in C^1([0,T)\times\mathbb{R})$ and $ \sigma \geq 0$ on $[0,T)\times\mathbb{R};$ 
    \item[(ii)] the equations \eqref{eq:carroll eqns i} are satisfied in the pointwise sense for all $(t,x) \in (0,T)\times \mathbb{R}.$
\end{itemize}
If, in addition, $(\sigma(0,x),\beta(0,x)) = (\sigma_0(x),\beta_0(x))$ for all $x \in \mathbb{R}$, then we say that $(\sigma,\beta)$ is a \emph{$C^1$ solution to the Cauchy problem with initial data $(\sigma_0,\beta_0)$}. If $(\sigma,\beta)$ is a $C^1$ solution on $(0,T)\times\mathbb{R}$ for all $T>0$, then we say it is a \emph{global $C^1$ solution}. 
\end{definition}

\begin{remark}
    Note that the choice to require $\sigma \geq 0$ in the above definition stems primarily from the need to make sense of the expression $\sigma^\gamma$ when $\gamma \notin \mathbb{Z}$. In the particular case of $\gamma = 3$ (the only \emph{odd} integer in the range $(1,3]$), note that this requirement is arbitrary in the sense that \eqref{eq:carroll eqns i} exhibits the symmetry $\sigma \mapsto - \sigma$. In the absence of physical intuition guiding the choice of sign (in contrast to the Galilean case), in this case a notion of solution with $\sigma \leq 0$ may be defined analogously.
\end{remark}

Our first main result is concerned with necessary and sufficient conditions for finite-time blow-up of classical solutions for the particular case $\gamma=3$. In terms of the terminology introduced in Definition \ref{def:compressive rarefactive}, this result states that a $C^1$ singularity forms in finite time if and only if the initial data are compressive somewhere.

\begin{theorem}[$C^1$ Blow-up Criterion for $\gamma=3$]\label{thm:blow up criterion gamma is 3}
    Let $(\sigma_0,\beta_0) \in L^\infty(\mathbb{R}) \cap C^1(\mathbb{R})$ satisfy 
    \begin{equation}\label{eq:initial data statement gamma is 3 blow up C1}
        \inf_{\mathbb{R}} (\sigma_0 - |\beta_0| ) > 0. 
    \end{equation}
    Then there exists a unique global $C^1$ solution $(\sigma,\beta)$ of \eqref{eq:carroll eqns i} with $\gamma=3$ and initial data $(\sigma_0,\beta_0)$ if and only if 
    \begin{equation}\label{eq:condition statement gamma is 3}
        (\beta_0 + \sigma_0)_x \leq 0 \quad \text{and} \quad ( \beta_0 - \sigma_0 )_x \leq 0 \quad \text{everywhere on } \mathbb{R}. 
        \end{equation}
    If condition \eqref{eq:condition statement gamma is 3} fails, then the solution ceases to be $C^1$ at the time 
   \begin{equation}\label{eq:blowup time gamma is 3 thm statement}
       T^* \defeq  \min\bigg\{ \inf_{\mathbb{R}}\frac{(\beta_0 + \sigma_0)^2}{(\beta_0+\sigma_0)_{x}} , ~ \inf_{\mathbb{R}} \frac{(\beta_0-\sigma_0)^2}{(\beta_0-\sigma_0)_{x}} \bigg\}. 
   \end{equation}
   Furthermore, for all $t \in [0,T^*)$, there hold the one-sided Lipschitz bounds 
   \begin{equation}\label{eq:one sided lip statement gamma is 3}
       \inf_{\mathbb{R}}(\beta \pm \sigma)_x (t,\cdot) \geq - \frac{\sup_{\mathbb{R}}(\sigma_0+|\beta_0|)^2}{t}. 
   \end{equation}
\end{theorem}

We emphasise that the theorem above implies that, given generic initial data satisfying $\inf_{\mathbb{R}} (\sigma_0 - |\beta_0| ) > 0$ but whose derivatives do not satisfy the sign requirements \eqref{eq:condition statement gamma is 3}, the solution ceases to be $C^1$ in finite time. Furthermore, given the invariant regions of \S \ref{sec:invariant}, such blow-up occurs without the formation of \emph{Carrollian liquescence} ($\sigma = 0$). Moreover, we will see in \S \ref{sec:invariant} that the initial condition \eqref{eq:initial data statement gamma is 3 blow up C1} implies that $\inf_\mathbb{R}(\sigma(t,\cdot) - |\beta(t,\cdot)| ) > 0$ for all subsequent times $t>0$, which preserves the strict hyperbolicity of the system; \textit{cf.}~\S \ref{sec:hyp}. The condition for non-blow-up given in \eqref{eq:condition statement gamma is 3} is equivalent to the \emph{everywhere rarefactive} condition of \Cref{def:compressive rarefactive}; see the expression for the eigenvalues in \eqref{eq:eigenvalues general case}.

For the general case $\gamma \in (1,3)$, we are also able to establish necessary and sufficient conditions for $C^1$ blow-up in finite-time. The following quantity will appear recurrently in our analysis 
\begin{equation}\label{eq:theta def}
    \theta \defeq  \frac{\gamma-1}{2}; 
\end{equation}
note that $\theta \in (0,1]$ for our admissible range of $\gamma$. We introduce the functions of the phase space variables 
\begin{equation}\label{eq:riemann invariant functions intro}
    w_1(\sigma,\beta) \defeq \beta + \frac{\sigma^\theta}{\theta}, \quad w_2(\sigma,\beta) \defeq \beta - \frac{\sigma^\theta}{\theta}; 
\end{equation}
we show in \S \ref{sec:hyp} that these are Riemann invariants of the system \eqref{eq:carroll eqns i}. In the remainder of the paper we will use the notation $w_j(0,\cdot) = w_j(\sigma_0,\beta_0)$ and $w_{jx} = \partial_x w_j$  ($j=1,2$). Once again, in terms of the terminology introduced in Definition \ref{def:compressive rarefactive}, our main result for the $\gamma \in (1,3)$ case states that a $C^1$ singularity forms in finite time if and only if the initial data are compressive somewhere. 

\begin{theorem}[$C^1$ Blow-up Criterion for $\gamma \in (1,3)$]\label{thm:blow up criterion general gamma}
    Let $(\sigma_0,\beta_0) \in L^\infty(\mathbb{R}) \cap C^1(\mathbb{R})$ satisfy the conditions\footnote{Note that this imposes $\inf_\mathbb{R} \sigma_0 > 0$.} 
    \begin{equation}\label{eq:initial data statement gamma general blow up C1}
    \inf_\mathbb{R} w_1(0,\cdot) > 0 > \sup_\mathbb{R} w_2(0,\cdot), 
\end{equation}
and 
\begin{align} \begin{split} \label{eq:initial data statement gamma general blow up C1 second}
        \big(\sup_\mathbb{R} w_1(0,\cdot)-\theta \inf_\mathbb{R} w_1(0,\cdot)\big) + (1+\theta)\sup_\mathbb{R} w_2(0,\cdot) &< 0 , \\
        (1+\theta)\inf_\mathbb{R}w_1(0,\cdot) - \big( \theta \sup_\mathbb{R} w_2(0,\cdot)-\inf_\mathbb{R}w_2(0,\cdot) \big) & > 0. 
        \end{split}
\end{align}
Then there exists a unique global $C^1$ solution $(\sigma,\beta)$ of \eqref{eq:carroll eqns i} with $\gamma \in (1,3)$ and initial data $(\sigma_0,\beta_0)$ \hspace{.5mm} if and only if 
    \begin{equation} \label{eq:condition statement gamma general}
        w_{1x}(0,x) \leq 0 \quad \text{and} \quad w_{2x}(0,x) \leq 0 \quad \text{for all } x \in \mathbb{R}. 
        \end{equation}
        If condition \eqref{eq:condition statement gamma general} fails, then, by denoting by $T^*$ the smallest time at which the solution ceases to be $C^1$, there hold for all $t \in [0,T^*)$ the one-sided Lipschitz bounds 
        \begin{equation}\label{eq:one sided lipschitz gamma general}
            \inf_{\mathbb{R}} w_{jx}(t,\cdot) \geq - \frac{C}{t} \qquad (j=1,2), 
        \end{equation}
        with the constant $C>0$ depending only on the initial data. 
\end{theorem}

We note in passing that the condition \eqref{eq:initial data statement gamma is 3 blow up C1} is equivalent to \eqref{eq:initial data statement gamma general blow up C1} for $\gamma = 3$. In the particular case $\gamma=3$, the additional assumption \eqref{eq:initial data statement gamma general blow up C1 second} is not required; the condition \eqref{eq:initial data statement gamma general blow up C1 second} on the initial data is required for $\gamma \in (1,3)$ to preserve the strict hyperbolicity of the system. Additionally, a characterisation of the first blow-up time for general $\gamma \in (1,3)$ when condition \eqref{eq:condition statement gamma general} fails is given in \S \ref{sec:blow up}; see equations \eqref{eq:blow up time bounds w1x} and \eqref{eq:blow up time bounds w2x}. We note also that the condition for non-blow-up given in \eqref{eq:condition statement gamma general} is equivalent to the \emph{everywhere rarefactive} condition of \Cref{def:compressive rarefactive}.

\subsection{Background on the Lax method for finite-time blow-up}\label{sec:lax background}

In the interest of completeness, in this section we give a brief summary of Lax's argument (\textit{cf.}~\cite{Lax}) for finite-time blow-ups in general strictly hyperbolic systems. The same essential strategy is used in \S \ref{sec:gamma 3} and \S \ref{sec:blow up} to obtain blow-ups for \eqref{eq:carroll eqns i}. 

Consider the system of conservation laws 
\begin{equation}\label{eq:generic system}
\begin{cases}
u_t + f(u,v)_x = 0, \\
v_t + g(u,v)_x = 0.
\end{cases}
\end{equation}
In what follows, we provide a formal argument yielding the finite-time blow-up of \eqref{eq:generic system}. The argument proceeds by obtaining a Riccati-type equation for the spatial derivatives of the Riemann invariants and therefore showing that these derivatives explode in finite-time; we divide the argument into three steps. 

\smallskip 

\noindent 1. \textit{Evolution of the Riemann invariants}. Provided the solution $(u,v)$ of \eqref{eq:generic system} is $C^1$, one may apply the chain rule to rewrite \eqref{eq:generic system} as 
$$\mathbf{v}_t+ \mathbf{A} \mathbf{v}_x = 0,$$
where $\mathbf{v}=(u,v)^\intercal$ and $\mathbf{A}$ is the Jacobian matrix \[\mathbf{A} = \begin{pmatrix} f_u & f_v \\ g_u & g_v \end{pmatrix}.\]
Provided the system is strictly hyperbolic, the matrix $\mathbf{A}$ admits two distinct real eigenvalues, which we denote by $\lambda$ and $\mu$. Correspondingly, we denote by $\mathbf{r}_\lambda, \mathbf{r}_\mu$ the right-eigenvectors of $\mathbf{A}$ associated to the eigenvalues $\lambda,\mu$. Assuming moreover the existence of Riemann invariants $z$ and $w$, which are functions of the phase space variables $(u,v)$ that satisfy 
\begin{equation*}
    \nabla_{(u,v)} w \cdot \mathbf{r}_\lambda = 0, \quad \nabla_{(u,v)} z \cdot \mathbf{r}_\mu = 0, 
\end{equation*}
one recovers (\textit{cf.}~\textit{e.g.}~\cite[\S 11.3.1, Theorem 1]{Evans}) the diagonalised system 
\begin{equation}\label{eq:RI general lax}
    \left\lbrace\begin{aligned}
        &z_t + \lambda z_x = 0, \\ 
        &w_t + \mu w_x = 0. 
    \end{aligned}\right.
\end{equation}
Following Lax's notation, we define the operators 
\[ \prime \defeq  \partial_t + \lambda \partial_x \quad \text{and} \quad \backprime\defeq \partial_t + \mu \partial_x. \] 
In this notation \eqref{eq:RI general lax} becomes
\begin{equation*}
    z' = w^\backprime = 0. 
\end{equation*}

\smallskip 

\noindent 2. \textit{Riccati-type equation for $z_x$}. By differentiating the equation for $z$ in \eqref{eq:RI general lax} with respect to $x$ and setting $\alpha = z_x$, we obtain 
\[  \alpha^\prime + \lambda_z \alpha^2 + \lambda_w w_x \alpha = 0.         \]Meanwhile, the equation for $w$ yields $0=w^\backprime = w^\prime + (\mu- \lambda) w_x$, from which we obtain the expression 
\begin{equation*}
    w_x =\frac{w^\prime}{\lambda- \mu}. 
\end{equation*}
Substituting into the equation for $\alpha$, we find 
\begin{equation} \label{argumentlax}  \alpha^{\prime} + \lambda_{z} \alpha^{2} + \Big(\frac{ \lambda_w}{\lambda- \mu}w^{\prime}\Big) \alpha =0.       \end{equation}
\vspace{3mm} 
We now obtain a Riccati equation for $\alpha$ by introducing a function $h$ satisfying \[ h_w = \frac{\lambda_w}{\lambda -\mu}; \]
for illustrative purposes we assume here that such a function $h$ exists. Then \eqref{argumentlax} rewrites as \[  \alpha^\prime + h^\prime \alpha + \lambda_z \alpha^2  =0,    \]\textit{i.e.} by setting $\tilde{\alpha} = e^h \alpha, a = e^{-h} \lambda_z$ we get the aforementioned Riccati-type equation 
\be \label{riccati eq} 
\tilde{\alpha}^\prime = - a \hspace{.5mm} \tilde{\alpha}^2. 
\ee

\smallskip 

\noindent 3. \textit{Blow-up along characteristic}. We employ a classical ODE argument to give conditions for the finite-time blow-up of \eqref{riccati eq}. Let $x^1(t)$ be the characteristic vector field associated to the initial-value problem 
\[  \left\lbrace \begin{aligned} &\frac{\text{d} x^1(t)}{\text{d} t} = \lambda,\\ 
&x^1(0) = x_0; \end{aligned}\right.   \]
we formally assume that such a $x^1$ exists. By solving \eqref{riccati eq} along this characteristic, we obtain
\[ \frac{1}{\tilde{\alpha}(t,x^1(t))} = \frac{1}{\tilde{\alpha}(0,x_0)}+ \int_0^t  a(s, x^1(s)) \d s.    \]
Provided there exists $t_* \in (0,\infty)$ such that \[ \int_0^{t_*} a(\sigma,x^1(\sigma)) \hspace{.5mm} \text{d}\sigma = - \frac{1}{\tilde{\alpha}(0,x_0)},      \]
$\limsup_{t \uparrow t_*}|\tilde{\alpha}(t,x^1(t))| = \infty$ and the solution ceases to be continuously differentiable in finite time. An identical analysis can be performed on the derivative of the other Riemann invariant, $w_x$, to again establish its blow-up at a finite time $t_*$. If such a finite $t_*$ does not exist for neither $z_x$ nor $w_x$, then the Riemann invariants are $C^1$ for all time and the solution is global.

\section{Hyperbolicity and Genuine Nonlinearity}\label{sec:hyp}

We recast the governing equations into a first-order hyperbolic system; for $\gamma=3$, the system can be written in conservative form with respect to the unknown vector $(\sigma,\beta)$, while for $\gamma \in (1,3)$ this is not the case (\textit{cf.}~\cite[\S 2.2]{Carrollcomp}). By applying the chain rule and defining $\bu \defeq  (\sigma,\beta)$, we find that \eqref{eq:carroll eqns i} may be rewritten as 
\begin{equation*}
    \bu_t + \frac{1}{\beta^2 - \sigma^{\gamma-1}}\left( \begin{matrix}
        \beta & -\sigma \\ 
        -\sigma^{\gamma-2}  & \beta
    \end{matrix} \right)  \bu_x = 0, 
\end{equation*}
\textit{i.e.} 
\begin{equation}\label{eq:non conserv hyp}
    \bu_t + M^{-1}  \bu_x = 0, 
\end{equation}
where 
\begin{equation*}
    M = \left(\begin{matrix} \beta & \sigma \\ \sigma^{\gamma-2} & \beta \end{matrix} \right). 
\end{equation*}
Recall that the system \eqref{eq:non conserv hyp}, which is equivalent to \eqref{eq:carroll eqns i} for $C^1$ solutions, is said to be \emph{strictly hyperbolic} if $M^{-1}$ admits two distinct real eigenvalues. We have the following result. 

\begin{lemma}[Hyperbolicity and Riemann invariants]\label{lem:hyp}
    The system \eqref{eq:carroll eqns i} is strictly hyperbolic in the region 
    \begin{equation*}
       \mathcal{H} \defeq  \Big\{ (\sigma,\beta) \in (0,\infty)\times\mathbb{R} : \, \beta \neq \pm \sigma^\theta, \, \sigma \neq 0 \Big\}. 
    \end{equation*}
Furthermore, the system is endowed with the Riemann invariants given in \eqref{eq:riemann invariant functions intro}, \textit{i.e.} 
    \begin{equation*}
    w_1 = \beta + \frac{\sigma^\theta}{\theta}, \quad w_2 = \beta - \frac{\sigma^\theta}{\theta}. 
\end{equation*}
\end{lemma}
\begin{proof}
\noindent 1. \textit{Hyperbolicity}. We begin by computing the eigenvalues of $M$, which are given by
\begin{equation*}
    \mu_1 = \beta - \sigma^{\theta}, \quad \mu_2 = \beta + \sigma^{\theta},
\end{equation*}
where $\theta$ was defined in \eqref{eq:theta def}. Furthermore, a set of (non-normalised) right-eigenvectors for $M$ is 
\begin{equation}\label{eq:eigenvectors general case}
    \br_1 = \left( \begin{matrix}
        1 \\ 
        -\sigma^{\theta-1}
    \end{matrix} \right)    , \quad \br_2 =  \left( \begin{matrix}
        1 \\ 
        \sigma^{\theta-1}
    \end{matrix} \right). 
\end{equation}
Elementary manipulations show that the matrix $M^{-1}$ has right-eigenvectors $\br_1,\br_2$, with corresponding eigenvalues given by 
\begin{equation}\label{eq:eigenvalues general case}
    \lambda_1 = \frac{1}{\beta - \sigma^\theta}   , \quad \lambda_2 =  \frac{1}{\beta+\sigma^\theta}, 
\end{equation}
and the first part of the result follows. 

\smallskip 

\noindent 2. \textit{Riemann invariants}. Direct computation shows that the formulas given by \eqref{eq:riemann invariant functions intro} satisfy 
\begin{equation*}
    \nabla_\bu w_j \cdot \br_j = 0 \qquad (j=1,2), 
\end{equation*}
which proves the second part of the lemma. 
\end{proof}

We also show that the system is genuinely nonlinear in the region of the phase space of interest. 

\begin{lemma}[Genuine Nonlinearity]
    The system \eqref{eq:carroll eqns i} is genuinely nonlinear in the region 
    \begin{equation*}
        \mathcal{H} =  \Big\{ (\sigma,\beta) \in (0,\infty)\times\mathbb{R} : \, \beta \neq \pm \sigma^\theta, \, \sigma \neq 0  \Big\}. 
    \end{equation*}
\end{lemma}
\begin{proof}
    We verify the condition $\nabla_{(\sigma,\beta)} \lambda_j \cdot \mathbf{r}_j \neq 0$ ($j=1,2$). Direct computation from \eqref{eq:eigenvalues general case} yields 
    \begin{equation*}
        \nabla_{(\sigma,\beta)}\lambda_1 = \frac{1}{(\beta - \sigma^\theta)^2}\left( \begin{matrix}
        \theta \sigma^{\theta-1} \\ 
        -1
    \end{matrix} \right) , \quad \nabla_{(\sigma,\beta)}\lambda_2 = \frac{1}{(\beta + \sigma^\theta)^2}\left( \begin{matrix}
        -\theta \sigma^{\theta-1} \\ 
        -1
    \end{matrix} \right), 
    \end{equation*}
    whence, using also \eqref{eq:eigenvectors general case}, we get 
    \begin{equation*}
        \nabla_{(\sigma,\beta)}\lambda_1 \cdot \mathbf{r}_1 = \frac{(1+\theta)}{(\beta - \sigma^\theta)^2}\sigma^{\theta-1} > 0 > -\frac{(1+\theta)}{(\beta + \sigma^\theta)^2}\sigma^{\theta-1} = \nabla_{(\sigma,\beta)}\lambda_2 \cdot \mathbf{r}_2 
    \end{equation*}
in $\mathcal{H}$, and the result follows. 
\end{proof}

We conclude this section by noting that, provided $(\sigma,\beta)$ is $C^1$ and $|\sigma|>0$, the system \eqref{eq:carroll eqns i} may be rewritten in Riemann invariant coordinates in the diagonalised form: 
\begin{equation}\label{eq:riemann invariants eqns after hyp}
   \left\lbrace \begin{aligned}
        &\partial_t w_1 + \lambda_2 \partial_x w_1 = 0, \\ 
        &\partial_t w_2 + \lambda_1 \partial_x w_2 = 0, 
    \end{aligned}\right. 
\end{equation}
with $\lambda_j$ ($j=1,2$) as per \eqref{eq:eigenvalues general case}; \textit{cf.}~\textit{e.g.}~\cite[\S 11.3.1, Theorem 1]{Evans}.

\section{Invariant Regions and Local Well-Posedness}\label{sec:invariant and local}

This section is devoted to establishing the local-in-time existence and uniqueness of $C^1$ solutions of \eqref{eq:carroll eqns i}, and to computing invariant regions of the phase space to which these solutions are constrained. We begin with the invariant regions in \S \ref{sec:invariant} below, and move on to the local-in-time well-posedness in \S \ref{sec:local}. 

\subsection{Invariant Regions}\label{sec:invariant}

In this section, we establish the regions of the phase space to which solutions are constrained, given admissible initial data. This analysis is required since the eigenvalues \eqref{eq:eigenvalues general case} are not well-defined along the curves $\{\beta = \pm \sigma^\theta\}$, whence one loses the hyperbolicity of the system; hyperbolicity is also lost on the locus $\{\sigma = 0 \}$. We will show that neither of these occur provided the initial data satisfies the assumptions outlined in Theorems \ref{thm:blow up criterion gamma is 3} and \ref{thm:blow up criterion general gamma}. Throughout this subsection, we assume the existence and uniqueness of a local-in-time $C^1$ solution; this is justified in \S \ref{sec:local}. 

\smallskip

The main result of this section is as follows; in order to simplify notation, we write, with a slight abuse of notation, $w_j(t,x)$ in place of $w_j(\sigma(t,x),\beta(t,x))$ ($j=1,2$).

\begin{prop}[Invariant Regions for $\gamma \in (1,3)$]\label{prop:invariant regions}
Let $\gamma \in (1,3)$ and $(\sigma,\beta)$ be a $C^1$ solution on $[0,T)\times\mathbb{R}$ with initial data $(\sigma_0,\beta_0)$ satisfying \eqref{eq:initial data statement gamma general blow up C1}, \textit{i.e.} 
\begin{equation*}
    \inf_\mathbb{R} w_1(0,\cdot) > 0 > \sup_\mathbb{R} w_2(0,\cdot), 
\end{equation*}
which imposes $\inf_\mathbb{R} \sigma_0 > 0$, and \eqref{eq:initial data statement gamma general blow up C1 second}, \textit{i.e.}
\begin{equation*}
       \begin{aligned}
           &\big(\sup_\mathbb{R} w_1(0,\cdot)-\theta \inf_\mathbb{R} w_1(0,\cdot)\big) + (1+\theta)\sup_\mathbb{R} w_2(0,\cdot) < 0, \\ 
           &(1+\theta)\inf_\mathbb{R}w_1(0,\cdot) - \big( \theta \sup_\mathbb{R} w_2(0,\cdot)-\inf_\mathbb{R}w_2(0,\cdot) \big) > 0. 
       \end{aligned} 
\end{equation*}
Then there holds $(j=1,2)$ 
\begin{equation}\label{eq:riemann invariants max min pple}
    \inf_{\mathbb{R}} w_j(0,\cdot) \leq w_j(t,x) \leq \sup_\mathbb{R} w_j(0,\cdot) \quad \text{for all } (t,x) \in [0,T)\times\mathbb{R}, 
\end{equation}
as well as 
\begin{equation}\label{eq:sigma cone}
  \sigma^\theta(t,x) - |\beta(t,x)| > 0 \quad \text{for all } (t,x) \in [0,T)\times\mathbb{R}. 
\end{equation}
\end{prop}

For clarity of exposition, we briefly summarise the underlying idea behind the proof of Proposition \ref{prop:invariant regions} with the following formal argument. By interpreting the diagonalised system \eqref{eq:riemann invariants eqns after hyp} as two separate scalar conservation laws, we expect the Maximum Principle to imply ($j=1,2$) 
\begin{equation}\label{eq:max pple formal}
    \underbrace{\inf_\mathbb{R} w_j(0,\cdot)}_{\eqdef m_j} \leq w_j(t,x) \leq \underbrace{\sup_\mathbb{R} w_j(0,\cdot)}_{\eqdef M_j} \quad \text{for all } (t,x) \in [0,T)\times\mathbb{R}, 
\end{equation}
provided the eigenvalues $\lambda_j$ ($j=1,2$) are well-defined along solution trajectories in the phase space; \textit{i.e.}~$\beta \pm \sigma^\theta \neq 0$. In terms of the Riemann invariants, 
\begin{equation*}
    \beta = \frac{1}{2}(w_1+w_2), \quad \sigma^\theta = \frac{\theta}{2}(w_1-w_2), 
\end{equation*}
whence, in view of \eqref{eq:max pple formal}, there holds 
\begin{equation*}
    \frac{1}{2}(m_1+m_2) \leq \beta \leq \frac{1}{2}(M_1+M_2), \quad \frac{\theta}{2}(m_1-M_2) \leq \sigma^\theta \leq \frac{\theta}{2}(M_1-m_2), 
\end{equation*}
and thus 
\begin{equation}\label{eq:first eigenvalue bounds}
  \frac{1}{2}\Big[ (m_1 - \theta M_1) + (1+\theta)m_2  \Big]   \leq \beta-\sigma^\theta \leq \frac{1}{2}\Big[  (M_1-\theta m_1) + (1+\theta)M_2   \Big], 
\end{equation}
while 
\begin{equation}\label{eq:second eigenvalue bounds}
  \frac{1}{2}\Big[ (1+\theta)m_1 - (\theta M_2-m_2)  \Big]  \leq \beta + \sigma^\theta \leq \frac{1}{2}\Big[ (1+\theta)M_1 + (M_2 - \theta m_2) \Big]. 
\end{equation}
In turn, we simultaneously impose 
\begin{equation}\label{eq:eigen impose}
    (M_1-\theta m_1) + (1+\theta)M_2 < 0, \quad (1+\theta)m_1 - (\theta M_2-m_2) > 0, 
\end{equation}
which is precisely \eqref{eq:initial data statement gamma general blow up C1 second}, such that there holds $\lambda_2^{-1} = \beta+\sigma^\theta > 0$ and $\lambda_1^{-1} = \beta-\sigma^\theta < 0$.

\begin{remark}[Compatibility of initial data requirements] 
Note that the conditions \eqref{eq:initial data statement gamma general blow up C1} and \eqref{eq:initial data statement gamma general blow up C1 second} can always be simultaneously imposed for any value of $\theta > 0$, \textit{e.g.} choose $m_1 = m > 0$, $M_1 = (1+\theta)m$, $m_2 = (1+\theta)(-m)$, and $M_2 = -m$, and observe that $m_2 < M_2 < 0 < m_1 < M_1$ as well as 
\begin{equation*}
    (M_1-\theta m_1) + (1+\theta)M_2 = -\theta m < 0, \quad  (1+\theta)m_1 - (\theta M_2-m_2) = \theta m > 0. 
\end{equation*}
In fact, one can realise the required inequalities \eqref{eq:eigen impose} with the softer conditions $m_1 = m > 0$, $M_1 = (1+\theta)m$, and 
\begin{equation}\label{eq:soften mj constraint}
-\Big(2+\theta - \frac{1}{1+\theta}\Big)m < m_2 < M_2 < -\frac{m}{1+\theta}; 
\end{equation}
where we observe that 
\begin{equation*}
    -\Big(2+\theta - \frac{1}{1+\theta}\Big) < -\frac{1}{1+\theta} \quad \text{for all } \theta > 0 \text{ i.e.~}\gamma > 1, 
\end{equation*}
so that the condition \eqref{eq:soften mj constraint} is reasonable. This in particular gives an open set of initial data.
\end{remark}

\begin{proof}[Proof of Proposition \ref{prop:invariant regions}]

\noindent 1. \textit{Initialisation}.     Given the choice of initial condition, referring also to \eqref{eq:first eigenvalue bounds}--\eqref{eq:second eigenvalue bounds}, since the solution is $C^1$ and $w_1,w_2$ are continuous functions of $(\sigma,\beta)$, there exists a time interval $[0,t_*)$ over which there holds 
    \begin{equation}\label{eq:short time interval}
      w_1(t,x) > 0 > w_2(t,x), \quad \underbrace{\big(\beta - \sigma^\theta\big)(t,x)}_{=\lambda_1^{-1}} < 0 < \underbrace{\big(\beta + \sigma^\theta\big)(t,x)}_{=\lambda_2^{-1}}  \quad \text{for all } (t,x) \in [0,t_*)\times\mathbb{R}. 
    \end{equation}
  Suppose for contradiction that $t_{*}$ is the smallest time for which there exists $x_* \in \mathbb{R}$ such that any of the following happen:
  \begin{itemize}
      \item[(i)] $w_1(t_{*},x_*) = 0$; 
      \item[(ii)] $w_2(t_{*},x_*) = 0$; 
      \item[(iii)] $\beta(t_*,x_*) - \sigma^\theta(t_*,x_*) = 0$; 
      \item[(iv)] $\beta(t_*,x_*) + \sigma^\theta(t_*,x_*) = 0$. 
  \end{itemize}
We will show in Steps 2 and 3 of the proof that such $t_*$ cannot exist by a classical ODE argument. 

    \smallskip 
    
  \noindent 2. \textit{Preservation of positivity of Riemann invariants}.  Observe that, on $[0,t_{*})\times\mathbb{R}$, the eigenvalues $\lambda_j$ ($j=1,2$) are well-defined and continuous functions of $(\sigma,\beta)$ as a consequence of \eqref{eq:short time interval}. In turn, the characteristic curves defined by the initial-value problems ($j=1,2$), 
    \begin{equation}\label{eq:char curves invariant regions}
        \left\lbrace \begin{aligned}
            &\frac{\der x^j(t;x_0)}{\der t} = \lambda_j(t;x_0), \\ 
            &x^j(0;x_0) = x_0, 
        \end{aligned} \right. 
    \end{equation}
   where we use the slight abuse of notation $\lambda_j(t;x_0) = \lambda_j(\sigma(t,x^j(t;x_0)),\beta(t,x^j(t;x_0)))$, are well-defined and continuously differentiable for all $t \in [0,t_{*})$ and all choices of $x_0 \in \mathbb{R}$. Furthermore, as a consequence of the Inverse Function Theorem, the sign conditions on $\lambda_j$ imply that the curves $x_0 \mapsto x^j(t;x_0)$ are $C^1$-diffeomorphisms of $\mathbb{R}$ for $t \in [0,t_{*})$. It follows from \eqref{eq:riemann invariants eqns after hyp}, which is satisfied as a pointwise equality for $C^1$ solutions, that for all $t \in [0,t_*)$ 
    \begin{equation}\label{eq:solve char invariant}
      \begin{aligned} 
            & \inf_{\mathbb{R}} w_1(0,\cdot) \leq \underbrace{w_1(t,x^2(t;x_0))}_{= w_1(0,x_0)} \leq \sup_{\mathbb{R}} w_1(0,\cdot), \\ 
            & \inf_{\mathbb{R}} w_2(0,\cdot) \leq \underbrace{w_2(t,x^1(t;x_0))}_{= w_2(0,x_0)} \leq \sup_{\mathbb{R}} w_2(0,\cdot). 
      \end{aligned}
    \end{equation}
As a consequence of $x_0 \mapsto x^j(t;x_0)$ being $C^1$-diffeomorphisms of $\mathbb{R}$ for $t \in [0,t_*)$, it follows that, for all $(t,x) \in [0,t_*)\times\mathbb{R}$, 
    \begin{equation*}
       \inf_{\mathbb{R}} w_1(0,\cdot) \leq w_1(t,x) \leq \sup_{\mathbb{R}} w_1(0,\cdot) , \quad \inf_{\mathbb{R}} w_2(0,\cdot) \leq w_2(t,x) \leq \sup_{\mathbb{R}} w_2(0,\cdot). 
    \end{equation*}
    Using the continuity of $w_j$ and the previous line, we find 
    \begin{equation*}
        w_1(t_*,x_*) = \lim_{t \uparrow t_*} w_1(t,x_*) \geq \inf_{\mathbb{R}} w_1(0,\cdot) > 0, \quad w_2(t_*,x_*) = \lim_{t \uparrow t_*} w_2(t,x_*) \leq \sup_{\mathbb{R}} w_2(0,\cdot) < 0, 
    \end{equation*}
    which contradicts possibilities (i) and (ii). 

    \smallskip 

  \noindent 3. \textit{Preservation of positivity of eigenvalues}. For possibilities (iii) and (iv), the computations \eqref{eq:first eigenvalue bounds}--\eqref{eq:second eigenvalue bounds} are justified on the interval $[0,t_*)$ by the previous argument. Thus, the continuity of $\sigma,\beta$ yields 
    \begin{equation*}
        \big(\beta+\sigma^\theta\big)(t_*,x_*) = \lim_{t\uparrow t_*} \big(\beta+\sigma^\theta\big)(t,x_*) \geq (1+\theta)\inf_\mathbb{R}w_1(0,\cdot) - \big( \theta \sup_\mathbb{R} w_2(0,\cdot)-\inf_\mathbb{R}w_2(0,\cdot) \big) > 0, 
    \end{equation*}
    and similarly 
    \begin{equation*}
        \big(\beta-\sigma^\theta\big)(t_*,x_*) = \lim_{t\uparrow t_*} \big(\beta-\sigma^\theta\big)(t,x_*) \leq \big(\sup_\mathbb{R} w_1(0,\cdot)-\theta \inf_\mathbb{R} w_1(0,\cdot)\big) + (1+\theta)\sup_\mathbb{R} w_2(0,\cdot) < 0, 
    \end{equation*}
    which contradicts possibilities (iii) and (iv). It follows that \eqref{eq:short time interval} holds for all $(t,x) \in [0,T) \times \mathbb{R}$, and thus we have proved \eqref{eq:sigma cone}. 

    \smallskip 

    \noindent 4. \textit{Preservation of initial positivity}: It follows from Steps 2 and 3 that, since \eqref{eq:short time interval} holds on all of $[0,T)\times\mathbb{R}$, the characteristic curves \eqref{eq:char curves invariant regions} are well-defined and continuously differentiable for all times $t \in [0,T)$ and all choices of $x_0 \in \mathbb{R}$. In turn, \eqref{eq:solve char invariant} is valid for all $t \in [0,T)$, with characteristic curves $x_0 \mapsto x^j(t;x_0)$ being $C^1$-diffeomorphisms of $\mathbb{R}$ for all $t \in [0,T)$, and we deduce 
    \begin{equation*}
       \inf_{\mathbb{R}} w_j(0,\cdot) \leq w_j(t,x) \leq \sup_{\mathbb{R}} w_j(0,\cdot)\quad \text{for all } (t,x) \in \mathbb{R}, 
    \end{equation*}
 ($j=1,2$) and therefore \eqref{eq:riemann invariants max min pple} is proved. 
\end{proof}

In the specific case $\gamma=3$, the proof and required conditions simplify considerably. Indeed, in this case $\lambda_1^{-1} = w_2$ and $\lambda_2^{-1} = w_1$. It therefore follows from the same argument as in the Proof of Proposition \ref{prop:invariant regions} that the initial condition \eqref{eq:initial data statement gamma general blow up C1}, \textit{i.e.} 
\begin{equation*}
    \inf_\mathbb{R} w_1(0,\cdot) > 0 > \sup_\mathbb{R} w_2(0,\cdot) 
\end{equation*}
is sufficient to ensure that $\sigma - |\beta| > 0$ on $[0,T)\times\mathbb{R}$. We therefore have the following corollary. 

\begin{corollary}[Invariant Regions for $\gamma=3$]\label{cor:invariant regions gamma is 3}
    Let $\gamma =3$ and $(\sigma,\beta)$ be a $C^1$ solution on $[0,T)\times\mathbb{R}$ with initial data $(\sigma_0,\beta_0)$ satisfying 
\begin{equation*}
    \inf_\mathbb{R} w_1(0,\cdot) > 0 > \sup_\mathbb{R} w_2(0,\cdot), 
\end{equation*}
which imposes $\inf_\mathbb{R} \sigma_0 > 0$. Then, there holds $(j=1,2)$ 
\begin{equation*}
    \inf_{\mathbb{R}} w_j(0,\cdot) \leq w_j(t,x) \leq \sup_\mathbb{R} w_j(0,\cdot) \quad \text{for all } (t,x) \in [0,T)\times\mathbb{R}, 
\end{equation*}
as well as 
\begin{equation*}
  \sigma(t,x) - |\beta(t,x)| > 0 \quad \text{for all } (t,x) \in [0,T)\times\mathbb{R}. 
\end{equation*}
\end{corollary}

\subsection{Local-in-time well-posedness}\label{sec:local}

In what follows, we establish the existence and uniqueness of a local-in-time solution of \eqref{eq:carroll eqns i}; we do so by showing local existence to \eqref{eq:riemann invariants eqns after hyp} and then mapping back to the original variables $(\sigma,\beta)$ by the chain rule. To begin, we recall the following standard result, which may be found, for example, in \cite{LeeYu64}.

\begin{lemma}\label{localexistence}Consider the Cauchy problem for the system
 \begin{equation*}
\begin{cases}
s_t+a_2(r,s) \hspace{.5mm}s_x=0, \\ r_t+a_1(r,s)\hspace{.5mm}r_x=0, \end{cases}\end{equation*}
where the initial data $(r_0,s_0) \in C^1(\mathbb{R})$ and the functions $a_i \in C^{1}(\mathbb{R}^2)$ are such that either of the following conditions are satisfied: 
\begin{itemize}
    \item[(i)] $\partial_r a_i$ and $\partial_s a_i$ are both non-negative everywhere; 
    \item[(ii)] $\partial_r a_i$ and $\partial_s a_i$ are both non-positive everywhere. 
\end{itemize}
Then there exists $T>0$ such that, on the domain $[0,T] \times \mathbb{R}$, there exists a unique $C^1$ solution. 
\end{lemma}

Using \Cref{localexistence}, we obtain the following result concerning local-in-time well-posedness in $C^1$ for the system \eqref{eq:carroll eqns i}. 

\begin{prop}[Local-in-time well-posedness]\label{prop:local well posed}
    Let $(\sigma_0,\beta_0)$ be $C^1$ initial data satisfying the condition \eqref{eq:initial data statement gamma general blow up C1}, \textit{i.e.}
        \begin{equation*}
            \inf_\mathbb{R} w_1(0,\cdot) > 0 > \sup_\mathbb{R} w_2(0,\cdot). 
        \end{equation*}
        Furthermore, if $\gamma \in (1,3)$, assume additionally that \eqref{eq:initial data statement gamma general blow up C1 second} holds, \textit{i.e.} 
        \begin{equation*}
       \begin{aligned} &\big(\sup_\mathbb{R} w_1(0,\cdot)-\theta \inf_\mathbb{R} w_1(0,\cdot)\big) + (1+\theta)\sup_\mathbb{R} w_2(0,\cdot) < 0, \\ 
       &(1+\theta)\inf_\mathbb{R}w_1(0,\cdot) - \big( \theta \sup_\mathbb{R} w_2(0,\cdot)-\inf_\mathbb{R}w_2(0,\cdot) \big) > 0. 
       \end{aligned}
\end{equation*}
Then, there exists $T>0$ such that, on the domain $[0,T]\times\mathbb{R}$, there exists a unique $C^1$ solution $(\sigma,\beta)$ of \eqref{eq:carroll eqns i}. Furthermore, for all $t \in [0,T]$, there holds 
\begin{equation}\label{eq:local exist preserve 1}
    \inf_\mathbb{R} w_1(t,\cdot) > 0 > \sup_\mathbb{R} w_2(t,\cdot), 
\end{equation}
and, if $\gamma \in (1,3)$, there also holds 
 \begin{equation}\label{eq:local exist preserve 2}
       \begin{aligned} &\big(\sup_\mathbb{R} w_1(t,\cdot)-\theta \inf_\mathbb{R} w_1(t,\cdot)\big) + (1+\theta)\sup_\mathbb{R} w_2(t,\cdot) < 0, \\ 
       &(1+\theta)\inf_\mathbb{R}w_1(t,\cdot) - \big( \theta \sup_\mathbb{R} w_2(t,\cdot)-\inf_\mathbb{R}w_2(t,\cdot) \big) > 0. 
       \end{aligned}
\end{equation}
\end{prop}

\begin{proof}
From Lemma \ref{lem:hyp}, in terms of the Riemann invariants, we have 
\begin{equation}\label{eq:eig in terms of ri}
    \lambda_1 = \frac{2}{w_1(1-\theta)+w_2(1+\theta)}, \quad \lambda_2 =\frac{2}{w_1(1+\theta)+w_2(1-\theta)}, 
\end{equation}
and thus 
\begin{equation}\label{eq:computations derivs eigenvalues wrt riemann invariants}
   \begin{aligned} &\lambda_{1 w_1} = \frac{-2(1-\theta)}{[w_1(1-\theta)+w_2(1+\theta) ]^2}  , \quad \lambda_{1 w_2} = \frac{-2(1+\theta)}{[ w_1(1-\theta)+w_2(1+\theta) ]^2} , \\ 
& \lambda_{2 w_1} = \frac{-2(1+\theta)}{[ w_1(1+\theta)+w_2(1-\theta) ]^2}, \quad \lambda_{2w_2} = \frac{-2(1-\theta)}{[ w_1(1+\theta)+w_2(1-\theta) ]^2}. 
   \end{aligned}
\end{equation}
It therefore follows that the conditions of Lemma \ref{localexistence} are satisfied, whence there exists $T>0$ for which there exists a unique $C^1$ solution $(w_1,w_2)$ on the domain $[0,T]\times\mathbb{R}$. Furthermore, due to Proposition \ref{prop:invariant regions}, the conditions \eqref{eq:local exist preserve 1} and \eqref{eq:local exist preserve 2} are satisfied on the entire interval $[0,T]$. As such, by defining 
\begin{equation*}
    \sigma \defeq  \Big( \frac{\theta}{2}(w_1 - w_2) \Big)^{\frac{1}{\theta}}, \quad \beta \defeq  \frac{1}{2}(w_1+w_2), 
\end{equation*}
we have that $(\sigma,\beta) \in C^1(\mathbb{R})$ is a local-in-time solution of \eqref{eq:carroll eqns i}; note that the condition \eqref{eq:local exist preserve 1} implies that $\sigma(t,x)>0$ for all $(t,x) \in [0,T]\times\mathbb{R}$, whence $\sigma^{\theta-1}$ is well-defined. For $\gamma \in (1,3)$, the additional condition \eqref{eq:local exist preserve 2} implies that the eigenvalues $\lambda_1$, $\lambda_2$ are well-defined and bounded on the domain $[0,T]\times\mathbb{R}$; for $\gamma=3$, just the condition \eqref{eq:local exist preserve 1} is sufficient. 
\end{proof}

\section{$C^1$ Blow-up Criterion for $\gamma=3$}\label{sec:gamma 3}

As already mentioned, the setting for $\gamma=3$ is particularly simple due to the fact that one can write 
\begin{equation*}
    \lambda_1 =  \frac{1}{w_2}  , \quad \lambda_2 = \frac{1}{w_1}, 
\end{equation*}
and the equations for the Riemann invariants completely decouple and read
\begin{equation}\label{eq:riemann inv eqn gamma is 3}
    \begin{aligned}
        \partial_t w_j + \frac{1}{w_j}\partial_x w_j = 0 \qquad (j=1,2). 
    \end{aligned}
\end{equation}
Armed with Proposition \ref{prop:invariant regions}, we are ready to give the proof of Theorem \ref{thm:blow up criterion gamma is 3}.

\begin{proof}[Proof of \Cref{thm:blow up criterion gamma is 3}]

We emphasise that, with the initial conditions as prescribed in the statement of the theorem, the characteristic curves are well-defined up to the first blow-up time; this is because the eigenvalues $\lambda_j$ are well-defined up to this time by virtue of Corollary \ref{cor:invariant regions gamma is 3}. Furthermore, local-in-time existence and uniqueness is guaranteed by Proposition \ref{prop:local well posed}. In turn, the computations that follow (which are written in terms of characteristic curves) are justified. 

\smallskip 

\noindent 1. \textit{Blow-up of $w_{1x}$}. We differentiate the first equation in \eqref{eq:riemann inv eqn gamma is 3} for $w_1$ with respect to $x$ and set $\alpha_1 = w_{1x}$. Recalling the operators $\prime = \partial_t + \lambda_2 \partial_x$ and $\backprime = \partial_t + \lambda_1 \partial_x$, we obtain the Riccati-type equation 
\[  \alpha_1^{\prime} = \frac{1}{(w_1)^2}{(\alpha_1)^2}.      \]
Integrating along the characteristic $t \mapsto x(t)$ chosen such that
\begin{equation*}
    \left\lbrace\begin{aligned}
          & \frac{\der x}{\der t} = \lambda_2(t,x(t)), \\ 
          &x(0) = x_0, 
    \end{aligned}\right. 
\end{equation*}
where we use the slight abuse of notation $\lambda_2(t,x(t)) = \lambda_2(\sigma(t,x(t)),\beta(t,x(t)))$, we have 
\[ \frac{1}{\alpha_1(t, x(t))}= \frac{1}{\alpha_1(0, x_0)} - \int_0^t \frac{1}{w_1(s, x(s))^2} \d s.     \]
Recalling the equation for the Riemann invariants, we have that $w_1$ is constant along the aforementioned characteristic, whence 
\begin{equation}\label{eq:C1 blow up w1 gamma is 3}
    \frac{1}{\alpha_1(t, x(t))} = \frac{1}{\alpha_1(0, x_0)} - \frac{t}{w_1(0, x_0)^2}. 
    \end{equation}
In turn, if there exists $x_0 \in \mathbb{R}$ such that $\alpha(0,x_0) > 0$, then $w_{1x}$ becomes infinite at time $t_*$ given by 
\begin{equation}\label{eq:w1 blow up time gamma is 3}
    t_* = \frac{w_1(0,x_0)^2}{w_{1x}(0,x_0)}; 
\end{equation}
and moreover the equality \eqref{eq:C1 blow up w1 gamma is 3} implies $\frac{1}{\alpha_1(t, x(t))} > 0 $ for all $t \in [0,t_*)$. On the other hand, if $\alpha_1(0,x_0) \leq 0$ for all $x_0 \in \mathbb{R}$, then \eqref{eq:C1 blow up w1 gamma is 3} implies the one-sided bound 
\begin{equation}\label{eq:one sided lipschitz w1x}
    w_{1x}(t,x(t)) \geq - \frac{w_1(0,x_0)^2}{t} \geq - \frac{\big(\sup_{x_0 \in \mathbb{R}} w_1(0,x_0)\big)^2}{t}, 
\end{equation}
in view of the initial data assumption $w_1(0,x_0) > 0$ for all $x_0 \in \mathbb{R}$.

\smallskip 

\noindent 2. \textit{Blow-up of $w_{2x}$}. An identical strategy for $w_2$ yields, by setting $\alpha_2 = w_{2x}$, 
    \begin{equation*}
        \alpha_2^\backprime = \frac{1}{(w_2)^2}(\alpha_2)^2. 
    \end{equation*}
Integrating along the characteristic $t \mapsto y(t)$ chosen such that
\begin{equation*}
    \left\lbrace\begin{aligned}
          & \frac{\der y}{\der t} = \lambda_1(t,y(t)), \\ 
          &y(0) = y_0, 
    \end{aligned}\right. 
\end{equation*}
where we use the slight abuse of notation $\lambda_1(t,y(t)) = \lambda_1(\sigma(t,y(t)),\beta(t,y(t)))$, we have 
\[ \frac{1}{\alpha_2(t, y(t))}= \frac{1}{\alpha_2(0, y_0)} - \int_0^t \frac{1}{w_2(s, y(s))^2} \d s,     \]
and again using that $w_2$ is constant along the aforementioned characteristic, we obtain 
\begin{equation}\label{eq:C1 blow up w2 gamma is 3}
    \frac{1}{\alpha_2(t, y(t))}= \frac{1}{\alpha_2(0, y_0)} - \frac{t}{w_2(0, y_0)^2}. 
\end{equation}
As before, if there exists $y_0 \in \mathbb{R}$ such that $\alpha_2(0,y_0)>0$, then $w_{2x}$ becomes infinite at time $t_{**}$ given by 
\begin{equation}\label{eq:w2 blow up time gamma is 3}
    t_{**} = \frac{w_2(0,y_0)^2}{w_{2x}(0,y_0)}. 
\end{equation}
Similarly to \eqref{eq:one sided lipschitz w1x}, the equality \eqref{eq:C1 blow up w2 gamma is 3} implies $\frac{1}{\alpha_2(t, x(t))} > 0 $ for all $t \in [0,t_*)$. On the other hand, if $\alpha_2(0,x_0) \leq 0$ for all $x_0 \in \mathbb{R}$, then \eqref{eq:C1 blow up w2 gamma is 3} implies the one-sided bound 
\begin{equation}\label{eq:one sided lipschitz w2x}
    w_{2x}(t,x(t)) \geq - \frac{w_2(0,y_0)^2}{t} \geq - \frac{\big(\inf_{y_0 \in \mathbb{R}} w_2(0,y_0)\big)^2}{t}, 
\end{equation}
in view of the initial data assumption $w_2(0,y_0) < 0$ for all $y_0 \in \mathbb{R}$. 

\smallskip 

\noindent 3. \textit{Conclusion}. We see from equations \eqref{eq:C1 blow up w1 gamma is 3}--\eqref{eq:C1 blow up w2 gamma is 3} that, if $w_{1x}(0,x) \leq 0$ an $w_{2x}(0,x) \leq 0$ for all $x\in\mathbb{R}$, then $\alpha_1$ and $\alpha_2$ (\textit{i.e.}~$w_{1x}$ and $w_{2x}$, respectively) are well-defined for all times, and the solution is global since then $\sigma_x$ and $\beta_x$ are both finite for all times and the equation\footnote{More precisely, once the $x$-derivatives are known to be continuous, the system \eqref{eq:carroll eqns i} is a non-degenerate linear system of two equations for $\sigma_t$ and $\beta_t$ with $C^0$ coefficients, implying that $\sigma_t$ and $\beta_t$ are continuous.} implies $\sigma_t$ and $\beta_t$ are also finite. If either of those sign conditions on the derivatives of the initial data fail at a point $x_0 \in \mathbb{R}$, then we have finite-time blow-up. The expression for the blow-up time \eqref{eq:blowup time gamma is 3 thm statement} follows from \eqref{eq:w1 blow up time gamma is 3} and \eqref{eq:w2 blow up time gamma is 3}. The estimates \eqref{eq:one sided lipschitz w1x} and \eqref{eq:one sided lipschitz w2x} give the one-sided Lipschitz bounds 
\begin{equation*}
    \inf_{\mathbb{R}}(\sigma \pm \beta)_x(t,\cdot) \geq - \frac{C}{t} 
\end{equation*}
with 
\begin{equation*}
   C= \max\Big\{ \big( \sup_{\mathbb{R}}(\beta_0+\sigma_0) \big)^2 , ~ \big( \inf_{\mathbb{R}}(\beta_0-\sigma_0) \big)^2 \Big\}, 
\end{equation*}
which is bounded above by the constant given in the statement of Theorem \ref{thm:blow up criterion gamma is 3}. This completes the proof.
\end{proof}

\section{$C^1$ Blow-up Criterion for $\gamma \in (1,3)$}\label{sec:blow up}

\begin{proof}[Proof of Theorem \ref{thm:blow up criterion general gamma}]
As in the proof of Theorem \ref{thm:blow up criterion gamma is 3}, we emphasise again that, with the initial conditions as prescribed in the statement of the theorem, the characteristic curves are well-defined up to the first blow-up time, and local-in-time existence and uniqueness is guaranteed by Proposition \ref{prop:local well posed}.

\smallskip

We recall the expressions \eqref{eq:eig in terms of ri} and \eqref{eq:computations derivs eigenvalues wrt riemann invariants} in terms of the Riemann invariants, and, for the convenience of the reader, write the formula 
\begin{equation*}
    \lambda_2 - \lambda_1 = \frac{-4\theta(w_1-w_2)}{[w_1(1+\theta)+w_2(1-\theta)][w_1(1-\theta)+w_2(1+\theta)]}, 
\end{equation*}
as this quantity will arise several times in the argument below. 

\smallskip

\noindent 1. \textit{Blow-up of $w_{1x}$}. We apply the aforementioned strategy to study the evolution of the derivatives of the Riemann invariants along characteristic curves. Denoting $\alpha_1\defeq  w_{1x}$ and differentiating the first equation in \eqref{eq:riemann invariants eqns after hyp} with respect to $x$, we get 
\begin{equation}\label{eq:alpha 1 prime dynamic first} 
    \alpha_1^{\prime} + \lambda_{2w_1}(\alpha_1)^2 + \lambda_{2w_2}w_{2x} \alpha_1 =0.  
    \end{equation}
  In order to study the final term on the left-hand side, we use the expression $w_2^{\backprime}= w_2^{\prime}+(\lambda_1-\lambda_2)w_{2x}$, whence the second equation in \eqref{eq:riemann invariants eqns after hyp} (\textit{i.e.}~$w_2^\backprime = 0$) yields 
\begin{equation*} 
w_{2x} = \frac{w_2^{\prime}}{\lambda_2 -\lambda_1}. 
\end{equation*}
As a result, \eqref{eq:alpha 1 prime dynamic first}  becomes 
\begin{equation*} \alpha_1^{\prime} + \lambda_{2w_1} (\alpha_1)^2 + \lambda_{2w_2}\frac{w_2^{\prime}}{\lambda_2-\lambda_1} \alpha_1 =0. 
\end{equation*}

As per the strategy outlined in \S \ref{sec:lax background}, we seek to introduce a function $h_1$ such that 
\begin{equation}\label{eq:h1 w2 requirement}
    h_{1 w_2} = \frac{\lambda_{2w_2}}{\lambda_2-\lambda_1}, 
\end{equation}
when, using that $w_1' = 0$ from \eqref{eq:riemann invariants eqns after hyp}, the previous dynamical equation for $\alpha_1$ becomes 
\begin{equation}\label{eq:alpha 1 prime dynamic second} 
        \alpha_1^{\prime} + \lambda_{2w_1} (\alpha_1)^2 + h_1' \alpha_1 =0. 
\end{equation}
Direct computation shows that a suitable choice of function $h_1$ satisfying \eqref{eq:h1 w2 requirement} is 
\begin{equation*}
    h_1 = -\frac{(1-\theta)}{2\theta} \log(w_1 - w_2) - \log\big( (1+\theta) w_1 + (1-\theta)w_2  \big). 
\end{equation*}
Note that the logarithms are well-defined since their arguments are strictly positive by virtue of the invariant regions (\textit{cf.}~Proposition \ref{prop:invariant regions}); indeed, recalling the notation $m_j = \inf_\mathbb{R} w_j(0,\cdot)$, $M_j = \sup_\mathbb{R} w_j(0,\cdot)$, we have 
\begin{equation}\label{eq:important lower bound w1 computations}
    \begin{aligned} (1+\theta) w_1(t,x) + (1-\theta)w_2(t,x)  &\geq (1+\theta) \inf_\mathbb{R} w_1(t,\cdot) + \inf_\mathbb{R} w_2(t,\cdot) - \theta \sup_\mathbb{R} w_2(t,\cdot) \\ 
    &\geq (1+\theta) m_1 - (\theta M_2 - m_2) \\ 
    &> 0, 
    \end{aligned}
\end{equation}
where we used \eqref{eq:riemann invariants max min pple} to obtain the penultimate line and the condition \eqref{eq:initial data statement gamma general blow up C1 second} (\textit{cf.}~\eqref{eq:eigen impose}) to obtain the final inequality. 
Multiplying \eqref{eq:alpha 1 prime dynamic second} by $e^{h_1}$ and setting $\tilde{\alpha}_1\defeq  e^{h_1}\alpha_1$, we obtain the Riccati-type equation 
\begin{equation}\label{riccatione}
        \tilde{\alpha}_1^{\prime} = -e^{-h_1}\lambda_{2w_1} (\tilde{\alpha}_1)^2. 
\end{equation}
Using \eqref{eq:computations derivs eigenvalues wrt riemann invariants} and the explicit form of $h_1$, the previous equation yields 
\begin{equation*}
    (\tilde{\alpha}_1^{-1})^{\prime} = -2(1+\theta)\frac{(w_1-w_2)^{\frac{1-\theta}{2\theta}}}{w_1(1+\theta)+w_2(1-\theta) }, 
\end{equation*}
whence integrating along the characteristic $t \mapsto x(t)$ chosen such that
\begin{equation*}
    \left\lbrace\begin{aligned}
          & \frac{\der x}{\der t} = \lambda_2(t,x(t)), \\ 
          &x(0) = x_0, 
    \end{aligned}\right. 
\end{equation*}
we get 
\begin{equation}\label{eq:int along charac w1x general gamma}
    \frac{1}{\tilde{\alpha}_1(t,x(t))} = \frac{1}{\tilde{\alpha}_1(0,x_0)}  - 2(1+\theta) \int_0^t \frac{\big[w_1(s,x(s)) - w_2(s,x(s)) \big]^{\frac{1-\theta}{2\theta}}}{ (1+\theta)w_1(s,x(s))+ (1-\theta)w_2(s,x(s)) } \d s, 
\end{equation}
where we used the explicit formula for $h_1$. Note that this integral is well-defined and strictly positive by virtue of the lower bound \eqref{eq:important lower bound w1 computations}, which implies that the integrand is bounded; indeed, 
\begin{equation}\label{eq:first bounds integrand w1x}
    0 < \frac{1}{(1+\theta) M_1 + M_2 - \theta m_2}     \leq \frac{1}{w_1(1+\theta)+w_2(1-\theta)} \leq \frac{1}{(1+\theta)m_1 - (\theta M_2 - m_2) } < +\infty, 
\end{equation}
and, noting that $m_2 < M_2 < 0$ from the initial data assumptions, 
\begin{equation}\label{eq:second bounds integrand w1x}
    0 < (m_1 - M_2)^{\frac{1-\theta}{2\theta}} \leq (w_1 - w_2)^{\frac{1-\theta}{2\theta}} \leq (M_1 - m_2)^{\frac{1-\theta}{2\theta}} < +\infty. 
\end{equation}
We deduce from \eqref{eq:int along charac w1x general gamma} that, if $\alpha_1(0,x_0) \leq 0$ for all $x_0$, then $w_{1x}$ is well-defined for all times. On the other hand, if there exists $x_0$ such that $\alpha_1(0,x_0) > 0$, then there is blow-up in finite time, where the blow-up time $t_*$ is bounded above and below as follows: 
\begin{equation}\label{eq:blow up time bounds w1x}
    \frac{(1+\theta)m_1 -(\theta M_2 - m_2)}{2(1+\theta) \tilde{\alpha}_1(0,x_0)(M_1-m_2)^{\frac{1-\theta}{2\theta}}} \leq t_* \leq \frac{(1+\theta)M_1 + M_2 - \theta m_2}{2(1+\theta) \tilde{\alpha}_1(0,x_0)(m_1-M_2)^{\frac{1-\theta}{2\theta}}}, 
\end{equation}
where we used the lower and upper bounds provided by \eqref{eq:first bounds integrand w1x}--\eqref{eq:second bounds integrand w1x}.

\smallskip 

\noindent 2. \textit{Blow-up for $w_{2x}$}. We take a derivative with respect to $x$ in the equation for $w_2$ in \eqref{eq:riemann invariants eqns after hyp}. Setting $\alpha_2 = w_{2x}$, we get 
\begin{equation*}
    \alpha_2^\backprime + \lambda_{1 w_2}(\alpha_2)^2 + \lambda_{1 w_1} w_{1x} \alpha_2 = 0. 
\end{equation*}
As before, $w_1^\backprime = w_1^\prime + (\lambda_1-\lambda_2) w_{1x}$, whence the first equation in \eqref{eq:riemann invariants eqns after hyp} (\textit{i.e.}~$w_1'=0$) yields 
\begin{equation*}
    w_{1x} = \frac{w_1^\backprime}{\lambda_1 - \lambda_2}. 
\end{equation*}
The dynamical equation for $\alpha_2$ therefore becomes 
\begin{equation}\label{eq:dynamic eqn alpha 2}
    \alpha_2^\backprime + \lambda_{1 w_2}(\alpha_2)^2 + \lambda_{1 w_1}\frac{w_1^\backprime}{\lambda_1 - \lambda_2} \alpha_2 = 0. 
\end{equation}
Let 
\begin{equation*}
    h_2 = -\frac{(1-\theta)}{2\theta}\log( w_1-w_2 ) -\log \big[ - \big( w_1(1-\theta)+w_2 (1+\theta) \big) \big], 
\end{equation*}
which is well-defined on account of the estimate 
\begin{equation}\label{eq:important upper bound w2 computations}
   \begin{aligned} (1-\theta)w_1(t,x) + (1+\theta)w_2(t,x) &\leq \sup_\mathbb{R} w_1(t,\cdot) - \theta \inf_\mathbb{R} w_1(t,\cdot) + (1+\theta) \sup_\mathbb{R} w_2(t,\cdot) \\ 
   &\leq M_1 - \theta m_1 + (1+\theta) M_2 \\ 
   &< 0, 
   \end{aligned}
\end{equation}
where we used \eqref{eq:riemann invariants max min pple} to obtain the penultimate line and the condition \eqref{eq:initial data statement gamma general blow up C1 second} (\textit{cf.}~\eqref{eq:eigen impose}) to obtain the final inequality; see Proposition \ref{prop:invariant regions}. Observe that $h_2$ satisfies 
\begin{equation*}
    h_{2 w_1} = \frac{\lambda_{1 w_1}}{\lambda_1 - \lambda_2}, 
\end{equation*}
whence, using that $w_2^\backprime=0$ from \eqref{eq:riemann invariants eqns after hyp}, equation \eqref{eq:dynamic eqn alpha 2} reads 
\begin{equation*}
    \alpha_2^\backprime + \lambda_{1 w_2}(\alpha_2)^2 + h_2^\backprime \alpha_2 = 0. 
\end{equation*}
As before, we define $\tilde{\alpha}_2 \defeq  e^{h_2}\alpha_2$ and obtain the Riccati-type equation 
\begin{equation*}
    \tilde{\alpha}_2^\backprime = - e^{-h_2} \lambda_{1 w_2} (\tilde{\alpha}_2)^2, 
\end{equation*}
or, using the explicit form of $h_2$,
\begin{equation*}
    (\tilde{\alpha}_2^{-1})^\backprime = -2(1+\theta)\frac{ (w_1-w_2)^{\frac{1-\theta}{2\theta}}}{-[(1-\theta) w_1 + (1+\theta) w_2]}. 
\end{equation*}
Integrating the previous equation along the characteristic $t \mapsto y(t)$ chosen such that 
\begin{equation*}
    \left\lbrace\begin{aligned}
          & \frac{\der y}{\der t} = \lambda_1(t,y(t)), \\ 
          &y(0) = y_0, 
    \end{aligned}\right. 
\end{equation*}
we get 
\begin{equation}\label{eq:int along charac w2x general gamma}
    \frac{1}{\tilde{\alpha}_2(t,y(t))} = \frac{1}{\tilde{\alpha}_2(0,y_0)}  - 2(1+\theta) \int_0^t \frac{\big[w_1(s,y(s)) - w_2(s,y(s)) \big]^{\frac{1-\theta}{2\theta}}}{-[ (1-\theta)w_1(s,y(s))+ (1+\theta)w_2(s,y(s)) ]} \d s. 
\end{equation}
Again, the above integral is well-defined and strictly positive by virtue of the upper bound \eqref{eq:important upper bound w2 computations}, which implies that the integrand is bounded; indeed, 
\begin{equation}\label{eq:first bounds integrand w2x}
   \begin{aligned}
       0 < \frac{1}{-[m_1 - \theta M_1 + (1+\theta)m_2  ]}     &\leq \frac{1}{-[w_1(1-\theta)+w_2(1+\theta)]} \\ 
       &\leq \frac{1}{-[M_1 - \theta m_1 + (1+\theta) M_2 ] } \\ 
       &< +\infty, 
   \end{aligned} 
\end{equation}
and, noting that $m_2 < M_2 < 0$ from the initial data assumptions, 
\begin{equation}\label{eq:second bounds integrand w2x}
    0 < (m_1 - M_2)^{\frac{1-\theta}{2\theta}} \leq (w_1 - w_2)^{\frac{1-\theta}{2\theta}} \leq (M_1 - m_2)^{\frac{1-\theta}{2\theta}} < +\infty. 
\end{equation}
We deduce from \eqref{eq:int along charac w2x general gamma} that, if $\alpha_2(0,y_0) \leq 0$ for all $y_0$, then $w_{2x}$ is well-defined for all times. On the other hand, if there exists $y_0$ such that $\alpha_1(0,y_0) > 0$, then there is blow-up in finite time, where the blow-up time $t_{**}$ is bounded above and below as follows: 
\begin{equation}\label{eq:blow up time bounds w2x}
    \frac{-[M_1 - \theta m_1 + (1+\theta) M_2 ] }{2(1+\theta) \tilde{\alpha}_1(0,x_0)(M_1-m_2)^{\frac{1-\theta}{2\theta}}} \leq t_{**} \leq \frac{-[m_1 - \theta M_1 + (1+\theta)m_2  ]}{2(1+\theta) \tilde{\alpha}_1(0,x_0)(m_1-M_2)^{\frac{1-\theta}{2\theta}}}, 
\end{equation}
where we used the lower and upper bounds provided by \eqref{eq:first bounds integrand w2x}--\eqref{eq:second bounds integrand w2x}. 

\smallskip 

\noindent 3. \textit{Conclusion}. We have seen from \eqref{eq:int along charac w1x general gamma} and \eqref{eq:int along charac w2x general gamma} that if both $w_{1x}(0,x)<0$ and $w_{2x}(0,x) < 0$ for all $x\in\mathbb{R}$, then $\beta_x$ and $\sigma^{\theta-1} \sigma_x$ are well-defined for all times; given the positivity of $\sigma$ and using the equation to obtain finiteness of the time derivatives, this implies that $(\sigma,\beta)$ is a global $C^1$ solution. Otherwise, if this sign condition on the derivatives of the initial data fails at a point $x_0 \in \mathbb{R}$, we have finite time blow-up, with blow-up time satisfying the estimates \eqref{eq:blow up time bounds w1x} and \eqref{eq:blow up time bounds w2x}. The one-sided Lipschitz bounds \eqref{eq:one sided lipschitz gamma general} are deduced from \eqref{eq:int along charac w1x general gamma}--\eqref{eq:second bounds integrand w1x} and \eqref{eq:int along charac w2x general gamma}--\eqref{eq:second bounds integrand w2x} as per Step 3 of the proof of Theorem \ref{thm:blow up criterion gamma is 3}.
\end{proof}

\vspace{0.5cm}

\noindent\textbf{Acknowledgements.} The present work was triggered during discussions at the \emph{Journ\'ees Relativistes de Tours}, organised by X. Bekaert, Y. Herfray, S. Solodukhin and M. Volkov, held at the Institut Denis Poisson in June 2023. The authors acknowledge the hospitality and financial support from the Faculty of Mathematics of the University of Cambridge for the visits of NA, MP, and SMS in 2024. NA gratefully acknowledges support by an H.F.R.I. grant for postdoctoral researchers (3rd
call, no. 7126). SMS also acknowledges the support of Centro di Ricerca Matematica Ennio De Giorgi.

\printbibliography

\end{document}